\documentclass[12pt]{article}
\usepackage{amsfonts}
\usepackage{amsmath}
\usepackage{amssymb}
\usepackage[mathscr]{eucal}
\usepackage{tabu}
\usepackage{fullpage}
\usepackage{mathrsfs}
\usepackage[usenames]{color}
\usepackage[dvipsnames]{xcolor}
\usepackage{hyperref}    

\def\eod{\vrule height 6pt width 5pt depth 0pt}
\newenvironment{proof}{\noindent {\bf Proof:} \hspace{.2em}}
{\hspace*{\fill}{\eod}}

\newcommand{\ceil}[1]{\left\lceil #1 \right\rceil}
\newcommand{\floor}[1]{\left\lfloor #1 \right\rfloor}

\newcommand{\exc}{\mathrm{exc}}

\newcommand{\etal}{\textit{et al.}}
\newcommand{\EXC}{\mathrm{EXC}}
\newcommand{\des}{\mathrm{des}}
\newcommand{\asc}{\mathrm{asc}}

\newcommand{\DES}{\mathrm{DES}}

\newcommand{\inv}{\mathrm{inv}}

\newcommand{\BerB}{\mathscr{B}}

\newcommand{\stat}{\mathrm{stat}}

\newcommand{\SSS}{\mathfrak{S}}
\newcommand{\SgnBDes}{\mathrm{SgnBDes}}
\newcommand{\BB}{\mathfrak{B}}

\newcommand{\DD}{\mathfrak{D}}

\newcommand{\SDD}{\mathrm{SgnD}}

\newcommand{\AAA}{\mathcal{A}}
\newcommand{\RR}{\mathbb{R}}

\newcommand{\SB}{\mathrm{SgnB}}

\newcommand{\AAE}{ \mathrm{AE}}
\newcommand{\SgnAAE}{ \mathrm{SgnAE}}

\newcommand{\BE}{ \mathrm{BExc}}
\newcommand{\BDes}{ \mathrm{BDes}}
\newcommand{\BExc}{ \mathrm{BExc}}

\newcommand{\sgnBE}{ \mathrm{SgnBExc}}

\newcommand{\sgnDE}{ \mathrm{SgnDExc}}

\newcommand{\Sm}{ \mathrm{Sum}}

\newcommand{\DE}{ \mathrm{DExc}}

\newcommand{\Negs}{\mathrm{negs}}

\newcommand{\pos}{ \mathrm{pos}}

\newtheorem{theorem}{Theorem}

\newtheorem{corollary}[theorem]{Corollary}

\newtheorem{remark}[theorem]{Remark}

\newtheorem{lemma}[theorem]{Lemma}

\newcommand{\ZZ}{ \mathbb{Z}}

\newcommand{\comment}[1]{}

\newcommand{\ob}[1]{\overline{#1}}

\begin{document}
\title{Eulerian Central Limit Theorems and Carlitz identities 
in positive elements of Classical Weyl Groups}

\author{Hiranya Kishore Dey\\
Department of Mathematics\\
Indian Institute of Technology, Bombay\\
Mumbai 400 076, India.\\
email: hkdey@math.iitb.ac.in
\and
Sivaramakrishnan Sivasubramanian\\
Department of Mathematics\\
Indian Institute of Technology, Bombay\\
Mumbai 400 076, India.\\
email: krishnan@math.iitb.ac.in
}

\maketitle

\section{Introduction}
\label{sec:intro}
For a positive integer $n$, let $[n] = \{1,2,\ldots,n\}$ and let 
$\SSS_n$ denote the symmetric group on $[n]$.  
Let $\stat: \SSS_n \mapsto \ZZ_{\geq 0}$ be a statistic on $\SSS_n$.
Let $p_{n,k} = | \{\pi \in \SSS_n: \stat(\pi) = k \} |$.
If we sample permutations uniformly at random from $\SSS_n$, we get  
the random variable $X_{\stat}$ which takes the non negative 
integral value $k$ 
with probability $p_{n,k}/|\SSS_n|$.  Sometimes, we will sample
from a subset $S \subseteq \SSS_n$, in which case the probability
that $X_{\stat}=k$ will have to be suitably modified.  

Let $p(t) = \sum_k p_k t^k$ be a polynomial with non negative 
coefficients and let $p(1) > 0$.  We will use the straightforward 
bijection between such polynomials and non negative random variables 
$X$ which take the value $k$ with probability $p_k/p(1)$.

Given a statistic $\stat$ and subsets $T_n \subseteq \SSS_n$, 
consider the sequence of non negative integers 
$s_{n,k} = |\{ \pi \in T_n: \stat(\pi) = k \} |$.  Let
$M_n = \max_{\pi \in T_n} \stat(\pi)$ be the maximum value that
$\stat(\pi)$ takes over $\pi \in T_n$.   Suppose 
$\Sm_n = \sum_{k=0}^{M_n} s_{n,k}$ is such that $\Sm_n > 0$.  
The array $\{ s_{n,k}: n \geq 1, 0 \leq k \leq M_n \}$,
is said to satisfy a {\it Central Limit
Theorem} (CLT henceforth) with mean $\mu_n$ and variance 
$\sigma_n^2$ if 
\begin{equation}
\label{eqn:defn_clt}
\lim_{n \to \infty} \sup_{x \in \RR} \Bigg \rvert 
\sum_{k \leq \floor{(x)_n}} \frac{s_{n,k}}{\Sm_n} - \Phi(x) \Bigg \rvert = 0,
\end{equation}
where $(x)_n = x\sigma_n + \mu_n$ and 
$\displaystyle \Phi(x) = \frac{1}{\sqrt{2 \pi}} \int_{-\infty}^x e^{-t^2/2} dt$ is 
the cumulative distribution function of the standard 
normal distribution $N(0,1)$.  When this happens, we also say 
that the statistic $\stat$ is {\it asymptotically normal} over the
set $S$. Canfield's article 
\cite{canfield-asymptotic_normality_enumeration} is a good reference
for background on CLT's in enumerative combinatorics.

Several {\it Eulerian statistics} in $\SSS_n$ are known.  In this
paper, we focus on the two most basic such statistics: the number 
of descents, $\des$ and the number of excedances, $\exc$.  
For 
$\pi = \pi_1,\pi_2,\ldots,\pi_n \in \SSS_n$, define 
$\DES(\pi) = \{ i \in [n-1]: \pi_i > \pi_{i+1} \}$ and
$\EXC(\pi) = \{ i \in [n-1]: \pi_i > i\}$. Let 
$\des(\pi) = |\DES(\pi)|$ and $\exc(\pi) = |\EXC(\pi)|$.  
Define $A_{n,k} = | \{ \pi \in \SSS_n: \des(\pi) = k \}|$ and 
let $\AAE_{n,k} = | \{ \pi \in \SSS_n: \exc(\pi) = k \}|$.  As 
we will need the generating function later, we define 

\begin{equation}
\label{eqn:Eulerian_poly_Sn_defn}
A_n(t) = \sum_{k=0}^{n-1}A_{n,k} t^k \hspace{.32 cm} \mbox{ and } \hspace{.32 cm}
\AAE_n(t) = \sum_{k=0}^{n-1}\AAE_{n,k} t^k.
\end{equation}

David and Barton in 
\cite[Pages 150-154]{david-barton-combin-chance} showed that
the random variable $X_{\des}$ is asymptotically normal over 
$\SSS_n$. See the papers 
\cite[Example 3.5]{bender-central-local-limit} by Bender and
\cite{kahle-stump} by Kahle and Stump as well.

\begin{theorem}[David and Barton]
\label{thm:clt-perm}
For positive integers $n$, the random variable $X_{\des}$ 
over $\SSS_n$ is asymptotically normal with mean 
$(n-1)/2$ and variance $(n+1)/12$.
\end{theorem}

Let $\AAA_n \subseteq \SSS_n$ denote the subset of positive
elements in $\SSS_n$.  We alternatively denote $\AAA_n$ as $\SSS_n^+$
and $\SSS_n - \AAA_n$ as $\SSS_n^-$.  
Let $A_{n,k}^+ = |\{ \pi \in \SSS_n^+: \des(\pi) = k \}|$ and 
$A_{n,k}^- = |\{ \pi \in \SSS_n^-: \des(\pi) = k \}|$.  
Define
\begin{equation}
\label{eqn:Eulerian_poly_An-pm_defn}
A_n^+(t) = \sum_{k=0}^{n-1} A_{n,k}^+t^k 
\hspace{0.32 cm} \mbox{ and } \hspace{0.32 cm}
A_n^-(t) = \sum_{k=0}^{n-1} A_{n,k}^-t^k. 
\end{equation}

Fulman, Kim, Lee and Petersen recently in 
\cite[Theorem 1.2]{fulman-kim-lee-petersen_joint-distrib-descents-sign} 
showed that the random variable $X_{\des}$ over $\SSS_n^{\pm}$
has a CLT.

\begin{theorem}[Fulman, Kim, Lee and Petersen]
\label{thm:fulman-kim-lee-petersen}
The distribution of the coefficients of $A_n^{\pm}(t)$ is asymptotically
normal as $n \to \infty$.  For $n \geq 4$, these numbers have 
mean $(n-1)/2$ and for $n \geq 6$, these numbers have variance $(n+1)/12$.
\end{theorem}

It is a well known result of MacMahon \cite{macmahon-book} 
that excedances and descents are equidistributed over $\SSS_n$.  
Thus, Theorem \ref{thm:clt-perm} works when the random variable 
$X_{\des}$ is replaced by the random variable $X_{\exc}$.
Though excedances and descents are equidistributed over $\SSS_n$,
it is easy to check when $n \geq 3$, that they are 
{\it not equidistributed} over $\AAA_n$.  Hence, a version of 
Theorem \ref{thm:fulman-kim-lee-petersen} for the 
random variable $X_{\exc}$ does not follow immediately. 
Thus, a natural question is about the existence of a 
CLT for $X_{\exc}$ on the set $\AAA_n$.
In this work, we give such a CLT.  
Let $\AAE_{n,k}^+ = |\{ \pi \in \SSS_n^+: \exc(\pi) = k \}|$ and 
$\AAE_{n,k}^- = |\{ \pi \in \SSS_n^-: \exc(\pi) = k \}|$.  Define

\begin{equation}
\label{eqn:Eulerian_poly_exc_An-pm_defn}
\AAE_n^+(t) = \sum_{k=0}^{n-1} \AAE_{n,k}^+t^k
\hspace{0.32 cm} \mbox{ and } \hspace{0.32 cm}
\AAE_n^-(t) = \sum_{k=0}^{n-1} \AAE_{n,k}^-t^k.
\end{equation}
One of our main results, proved in subsection
\ref{subsec:main_lemma_typea} of this paper is the following.

\begin{theorem}
\label{thm:clt_exc_plus_minus}
The distribution of the coefficients of $\AAE_n^{\pm}(t)$ is asymptotically
normal as $n \to \infty$.  For $n \geq 3$, these numbers have 
mean $(n-1)/2$ and for $n \geq 4$, these numbers have variance $(n+1)/12$.
\end{theorem}

$\SSS_n$ is a Coxeter group and so are two other families 
$\BB_n$ and $\DD_n$.   The book \cite{bjorner-brenti} by 
Bj{\"o}rner and Brenti is a good reference for the combinatorics
of these groups.
Hence in both $\BB_n$ and $\DD_n$, there is a natural 
notion of descent.
We denote the descent statistic in type B and type D Coxeter 
groups as $\des_B$ and $\des_D$ 
(see definitions in Sections 
\ref{sec:typeb} and \ref{sec:typed}) respectively.

Define 
$B_{n,k} = | \{ \pi \in \BB_n: \des_B(\pi) = k \}|$
and 
$D_{n,k} = | \{ \pi \in \DD_n: \des_D(\pi) = k \}|$.
Define 

\begin{equation}
\label{eqn:Eulerian_poly_B_and_Dn_defn}
B_n(t) = \sum_{k=0}^n B_{n,k} t^k 
\hspace{0.32 cm} \mbox{ and } \hspace{0.32 cm}
D_n(t) = \sum_{k=0}^n D_{n,k} t^k.
\end{equation}

Kahle and Stump in \cite{kahle-stump}, recently computed the 
first two moments and gave a CLT for the appropriate 
$X_{\des}$ over $\BB_n$ and $\DD_n$.

\begin{theorem}[Kahle and Stump]
\label{thm:clt-kahle-stump}
For positive integers $n$, the random variable $X_{\des_B}$ over $\BB_n$ is 
asymptotically normal with mean $n/2$ and variance $(n+1)/12$.
For positive integers $n$, the random variable $X_{\des_D}$ over $\DD_n$ is 
asymptotically normal with mean $n/2$ and variance $(n+2)/12$.
\end{theorem}

Coxeter groups also have a natural notion of length which 
in $\SSS_n$, $\BB_n$ and $\DD_n$ will be  denoted as 
$\ell$, $\ell_B$ and $\ell_D$ respectively.  
In $\SSS_n$, it is easy to see that 
$ \AAA_n = \SSS_n^+ 
= \{ \pi \in \SSS_n: \ell(\pi) \equiv 0 \> ( \!\!\!\! \mod 2) \}$. 
Motivated by this, define 
$\BB_n^+ = \{ \pi \in \BB_n: \ell_B(\pi) \equiv 
0 \> (\!\!\!\! \mod 2) \}$ and $\DD_n^+ = \{ \pi \in \DD_n:
\ell_D(\pi) \equiv 0 \> ( \!\!\!\! \mod 2) \}$.
Similarly define 
$\BB_n^- = \{ \pi \in \BB_n: \ell_B(\pi) \equiv 
1 \> (\!\!\!\! \mod 2) \}$ and $\DD_n^- = \{ \pi \in \DD_n: 
\ell_D(\pi) \equiv 1 \> (\!\!\!\! \mod 2) \}$.  We will define
$(\BB_n-\DD_n)^{\pm}$ as well.  For $\pi \in \BB_n - \DD_n$, 
we choose $\inv_D(\pi)$ as the exponent of $-1$ to define
sign.  Had we chosen $\inv_B(\pi)$ for this purpose, 
Remark \ref{rem:invb_invd_ddn} makes it clear that the sets 
$(\BB_n - \DD_n)^{\pm}$ will be the same, but have their
names swapped.

\vspace{2 mm}

Let $B_{n,k}^+ = |\{ \pi \in \BB_n^+: \des_B(\pi) = k \}|$ and 
$B_{n,k}^- = |\{ \pi \in \BB_n^-: \des_B(\pi) = k \}|$.  
Define
\begin{equation}
\label{eqn:Eulerian_poly_Bn-pm_defn}
B_n^+(t) = \sum_{k=0}^n B_{n,k}^+t^k 
\hspace{0.32 cm} \mbox{ and } \hspace{0.32 cm}
B_n^-(t) = \sum_{k=0}^n B_{n,k}^-t^k. 
\end{equation}

For both these groups, CLT results for the random variable 
$X_{\des_B}$ (and $X_{\des_D}$) on the sets $\BB_n^{\pm}$ 
(and $\DD_n^{\pm}$) are not known to the 
best of our knowledge.   In Theorems 
\ref{thm:typeb-clt_exc_plus_minus} and
\ref{thm:typed-clt_exc_plus_minus},
we give such results.  Thus, these can be considered as
type B and type D counterparts of 
Theorem \ref{thm:fulman-kim-lee-petersen}
and are proved in Sections \ref{sec:typeb} and 
\ref{sec:typed} respectively.

In \cite{boroweic-mlotkowski-new-eulerian-type-d}, Borowiec 
and M\l{}otkowski enumerated the statistic $\des_B$ over 
$\DD_n$ and called the generating function a {\sl new type D Eulerian 
polynomial}. 
See Section \ref{sec:boroweic_mlot_refine} for definitions.  
From this polynomial, one can get CLT 
results (see Remark \ref{rem:clt_boroweic_follows}) for 
this variant.  
Our first contribution in this paper on this theme is Theorem 
\ref{thm:main_theorem_signed_desB_enumerate_over_type_d},
a signed enumeration result, where we enumerate $\des_B$ over $\DD_n$
(and over $\BB_n - \DD_n$) with sign taken into account.  
Using this result, in 
Theorem \ref{thm:typed-clt_bdes_plus_minus}, we give
a CLT when one considers the random variable $X_{\des_B}$ over 
$\DD_n^{\pm}$ (and over $(\BB_n -\DD_n)^{\pm}$).  
These results  are presented in Section \ref{sec:boroweic_mlot_refine}.
We summarize known CLT results and the new ones in the 
table below.  
In the table, we give the statistic rather than the random variable 
as if one writes the statistic as a subscript, it appears in a smaller 
font and is hence more difficult to read.

\vspace{3 mm}

\begin{tabu}{c|l|l}
\hline
Set& Descents & Excedance 
\\ \tabucline[2pt]{-}
$\SSS_n$ & For $\des$ see David and Barton \cite{david-barton-combin-chance}.
& For $\exc$, a CLT follows \\
& & from \cite{david-barton-combin-chance} and MacMahon \cite{macmahon-book}. \\ \hline
$\SSS_n^{\pm}$ & For $\des$ see Fulman $\etal$ \cite{fulman-kim-lee-petersen_joint-distrib-descents-sign}. 
& For $\exc$ see Theorem \ref{thm:clt_exc_plus_minus}.
\\ \tabucline[2pt]{-}
$\BB_n$ & For $\des_B$, see Kahle and Stump \cite{kahle-stump}. & For 
$\exc_B$, result  follows \\
& & from \cite{kahle-stump} and Brenti \cite{brenti-q-eulerian-94}.
\\  \hline
$\BB_n^{\pm}$ & For $\des_B$ see 
Theorem \ref{thm:typeb-clt_exc_plus_minus}.
 & For $\exc_B$, see
Theorem \ref{thm:typeb-clt_exc_plus_minus}.
\\ \tabucline[2pt]{-}
$\DD_n$ & For $\des_D$, see Kahle and Stump \cite{kahle-stump}. & 
For $\exc_D$, see Remark \ref{rem:typed-exc-done}. \\ \hline 
$\DD_n^{\pm}$ & For $\des_D$, see Theorem \ref{thm:typed-clt_des_plus_minus}.
& For $\exc_D$, see Theorem \ref{thm:typed-clt_exc_plus_minus}.
\\ \tabucline[2pt]{-}
$\DD_n$ & For $\des_B$, see Remark \ref{rem:clt_boroweic_follows}. 
&  For $\exc_B$, see Remark \ref{rem:boro_mlot-exc_not_new}.  \\ \hline
$\DD_n^{\pm}$ & 
For $\des_B$, see Theorem 
\ref{thm:typed-clt_bdes_plus_minus}. & 
For $\exc_B$, see Remark \ref{rem:boro_mlot-exc_not_new}. 
\\ 
\end{tabu}

\subsection{Carlitz type identities}
The following famous powerseries identity involving the Eulerian polynomial 
is attributed to Carlitz, though MacMahon's book 
\cite[Vol 2, Chap IV, pp 211]{macmahon-book} contains
this.  See \cite[Corollary 1.1]{petersen-eulerian-nos-book} 
in the book by Petersen as well.

\begin{theorem}[Carlitz]
\label{thm:carlitz-identity}
Let $A_n(t)$ be the Eulerian polynomial as defined in 
\eqref{eqn:Eulerian_poly_Sn_defn}.  For positive integers $n$,
\begin{equation}
\label{carlitz-original}
\frac{A_n(t)}{(1-t)^{n+1}} = \sum_{k \geq 0} (k+1)^nt^k.
\end{equation}
\end{theorem}

Fulman, Kim, Lee and Petersen in 
\cite[Theorem 1.1]{fulman-kim-lee-petersen_joint-distrib-descents-sign} 
gave the following refinement  of Theorem \ref{thm:carlitz-identity}
involving $A_n^{\pm}(t)$.

\begin{theorem}[Fulman, Kim, Lee and Petersen]
\label{thm:carlitz-identity_des_typeA}
Let $A_n^{\pm}(t)$ be the restricted versions of the descent based 
Eulerian polynomial defined in
\eqref{eqn:Eulerian_poly_An-pm_defn}.  For positive integers $n$,
\begin{equation}
\label{eqn:des_based_pm_carlitz_An}
\frac{A_n^{\pm}(t)}{(1-t)^{n+1}} = \sum_{k \geq 0}
\left( \frac{ (k+1)^n \pm (k+1)^{\ceil{n/2}} }{2}\right) t^k.
\end{equation}
\end{theorem}

In Section \ref{sec:powerseriesidentities} of
this paper, we prove the following similar but different identity 
involving the polynomial $\AAE_n^{\pm}(t)$.

\begin{theorem}
\label{thm:carlitz-identity_exc_typeA}
Let $\AAE_n^{\pm}(t)$ be the restricted versions of 
the excedance based Eulerian polynomial defined in
\eqref{eqn:Eulerian_poly_exc_An-pm_defn}.  For positive integers $n$,
\begin{equation}
\label{eqn:exc_based_pm_carlitz_An}
\frac{\AAE_n^{\pm}(t)}{(1-t)^{n+1}} = \sum_{k \geq 0}
\left( \frac{ (k+1)^n \pm (k+1)}{2} \right)t^k.
\end{equation}
\end{theorem}

Theorem \ref{thm:carlitz-identity_des_typeA} and 
Theorem \ref{thm:carlitz-identity_exc_typeA} are refinements
of Theorem \ref{thm:carlitz-identity} as 
summing up the two equations in 
\eqref{eqn:des_based_pm_carlitz_An} and
\eqref{eqn:exc_based_pm_carlitz_An} gives us 
\eqref{carlitz-original}.  
Here too, we consider type B and D variants.  In Theorems 
\ref{thm:carlitz-identity_pm_des_typeB} and
\ref{thm:carlitz-identity_pm_des_typeD}
we give type B and 
type D counterparts of Theorem \ref{thm:carlitz-identity_des_typeA}.
Brenti proved a type D Carlitz identity by giving a 
recurrence relation (see 
\cite[Corollary 4.8]{brenti-q-eulerian-94})
between the polynomials $B_n(t), D_n(t)$ and $A_{n-1}(t)$.
This recurrence was also proved by Stembridge in 
\cite[Lemma 9.1]{stembridge-perm-rep-weyl-grp-cohomology-toric-variety}.
Our proof of Theorem \ref{thm:carlitz-identity_pm_des_typeD}
refines this recurrence by giving two recurrences (see Lemma
\ref{lem:stembridge-refine}). 
Borowiec and M\l{}otkowski in 
\cite[Proposition 4.6]{boroweic-mlotkowski-new-eulerian-type-d},
also gave a Carlitz type identity involving their new type D 
Eulerian polynomial. 
In Theorem \ref{thm:carlitz-identity_pm_boroweic_typeD},
we refine their result by giving signed versions.
We prove all these results in Section \ref{sec:powerseriesidentities}.

\section{A lemma and the type A result}
\label{sec:our_results}

We start with the following lemma, which is implicit in the 
second proof of
\cite[Theorem 1.2]{fulman-kim-lee-petersen_joint-distrib-descents-sign}.
For positive integers $n$, let 
$F_n(t)= \sum_{k=0}^n f_{n,k}t^k$ and $G_n(t)=\sum_{k=0}^n g_{n,k}t^k$ 
be sequences of polynomials with $f_{n,k}, g_{n,k} \geq 0$ for all $k$.
Further, for all positive integers $n$, let $F_n(1) > 0$ and $G_n(1) > 0$.
We will consider the random variable $X_f^n$ (and $X_g^n$) which takes 
the value $k$ with probability $ \displaystyle \frac{f_{n,k}}{F_n(1)}$ 
\Big(and
$\displaystyle \frac{g_{n,k}}{G_n(1)}$ respectively\Big).  The following lemma 
gives a condition for moments of $X_f^n$ and $X_g^n$ to be identical.

\begin{lemma}
\label{lem:connection_between_all_and_plus}
Let $F_n(t)$ and $G_n(t)$ be as described above.  Suppose we have 
a sequence of polynomials $H_n(t)$, a non-zero real number $\lambda$ and a 
sequence of positive integers $\ell_n$ such that 
\begin{equation}
\label{eqn:diff_by_1-t-power}
F_n(t)= \lambda G_n(t) \pm (1-t)^{\ell_n} H_n(t).
\end{equation}

Then, for $r<\ell_n$, the $r$-th moment of $X_f^n$
is identical to the $r$-th moment of $X_g^n$.
\end{lemma}
\begin{proof}
We work with the $r$-th factorial moment instead. The main idea is to 
differentiate the relevant generating function $r$ times and 
then set $t=1$.   We first assume $\ell_n \geq 2$ and 
show that the first moments of $X_f^n$ and $X_g^n$ are equal.
Differentiating \eqref{eqn:diff_by_1-t-power}, we get 
\begin{equation}
\label{eqn:differentating1time}
\frac{dF_n(t)}{dt}  =  \lambda \frac{dG_n(t)}{dt}
 \pm \left(  (1-t)^{\ell_n} \frac{dH_n(t)}{dt} + \ell_n(1-t)^{\ell_n-1}H_n(t) \right).
\end{equation}

Setting $t=1$ in \eqref{eqn:differentating1time}, we get  
$\displaystyle \frac{dF_n(t)}{dt}\Bigg \rvert_{t=1}  =  
\lambda \frac{dG_n(t)}{dt}\Bigg \rvert_{t=1}.$  
That is $F'(1) = \lambda G'(1)$.  Here, we have used $\ell_n \geq 2$.
From \eqref{eqn:diff_by_1-t-power}, clearly,  $F(1) = \lambda G(1)$.
Since the first order factorial
moment of $X_f^n$ is $F'(1)/F(1)$, the first order factorial moment of 
$X_f^n$ is identical to the first order factorial moment of
$X_g^n$.
We continue 
differentiating upto $r$ times and then setting $t=1$.  For
$r < \ell_n$, we will have
$\displaystyle \frac{d^rF_n(t)}{dt^r} \Bigg \rvert_{t=1}  = 
\frac{d^rG_n(t)}{dt^r} \Bigg \rvert_{t=1}.$

\vspace{2 mm}

By a similar argument, when $r < \ell_n$, the $r$-th factorial 
moment of $X_f^n$  
is identical to  the $r$-th factorial moment of $X_g^n$.
The proof follows from the fact that the moments of two 
distributions are identical if and only if their factorial moments 
are identical. 
\end{proof}

\subsection{Proof of Theorem \ref{thm:clt_exc_plus_minus}}
\label{subsec:main_lemma_typea}

We will need the following signed excedance enumeration result 
in $\SSS_n$.   Let

\begin{equation}
\label{eqn:sgn_exc}
\SgnAAE_n(t) = \sum_{\pi \in \SSS_n} (-1)^{\inv(\pi)} t^{\exc(\pi)}.
\end{equation}

Mantaci (see \cite{mantaci-thesis, mantaci-jcta-93})
showed the following interesting result.  Sivasubramanian in 
\cite{siva-exc-det} gave an alternate proof of Mantaci's result by evaluating 
the determinant of appropriately defined $n \times n$ matrices.

\begin{theorem}[Mantaci]
\label{thm:mantaci}
For positive integers $n \geq 1$,  $\SgnAAE_n(t) = (1-t)^{n-1}$.
\end{theorem}

\begin{proof}(Of Theorem \ref{thm:clt_exc_plus_minus})
For $n \geq 1$, we clearly have 
\begin{equation}
\label{eqn:factor_reln_typeA}
\AAE_n^{\pm}(t) = \frac{1}{2} \Big( \AAE_n(t) \pm (1-t)^{n-1}\Big).
\end{equation}
Thus, by Lemma \ref{lem:connection_between_all_and_plus}, the first
$n-2$ moments of $\AAE_n^{\pm}(t)$ are identical to the 
first $n-2$ moments of $\AAE_n(t)$.  Since $\AAE_n(t) = A_n(t)$ and 
$A_n(t)$ is asymptotically normal, by the method of moments, 
$\AAE_n^{\pm}(t)$ is also asymptotically normal.  Further, they
have the same expected value as $A_n(t)$ when $n \geq 3$ and the 
same variance as $A_n(t)$ when $n \geq 4$.  The proof is complete.
\end{proof}

\section{Type B Coxeter Groups}
\label{sec:typeb}

It is known that $\BB_n$ can be thought as the 
group of permutations $\pi$ 
of the set $[\pm n] = \{-n,-(n-1),\ldots,-1,1,2,\ldots,n \}$ which
satisfy $\pi(-i) = -\pi(i)$ for $1 \leq i \leq n$.  
See the 
book by Bj{\"o}rner and Brenti \cite[Chapter 8]{bjorner-brenti}.
Clearly,
we only need $\pi(i)$ for $1 \leq i \leq n$ to know $\pi \in \BB_n$.
We denote $-i$ alternatively as $\ob{i}$ as well.
For $\pi \in \BB_n$, let $\Negs(\pi) = |\{ i \in [n]: \pi(i) < 0\}|$
be the number of negative elements in the image of $\pi(i)$ for
$i \in [n]$.

For a positive integer $n$, define $[n]_0 = \{0,1,\ldots,n\}$.
For 
$\pi = \pi_1,\pi_2,\ldots,\pi_n \in \BB_n$, let $\pi_0 = 0$. 
Define 
$\DES_B(\pi) = \{ i \in [n-1]_0: \pi_i > \pi_{i+1} \}$ and let
$\des_B(\pi) = |\DES_B(\pi)|$.
Define
$B_{n,k} = | \{ \pi \in \BB_n: \des_B(\pi) = k \}|$.
Following Brenti's definition of excedance from \cite{brenti-q-eulerian-94}, 
let  $\EXC_B(\pi)= \{ i \in [n]: \pi_{|\pi(i)|} > \pi_ i\} \cup
\{ i \in [n]: \pi_i =- i\}$ and let 
$\exc_B(\pi) = |\EXC_B(\pi)|$.  
Define $\BE_{n,k} = | \{ \pi \in \BB_n: \exc_B(\pi) = k \}|$.  
Let

\begin{equation}
\label{eqn:Eulerian_poly_Bn_defn}
B_n(t) = \sum_{k=0}^nB_{n,k} t^k \hspace{0.32 cm} \mbox{ and } \hspace{0.32 cm}
\BE_n(t) = \sum_{k=0}^n\BE_{n,k} t^k.
\end{equation}

Let 
$B_{n,k}^+ = |\{ \pi \in \BB_n^+: \des_B(\pi) = k \} |$ and
$B_{n,k}^- = |\{ \pi \in \BB_n^-: \des_B(\pi) = k \} |$.  Define
$B_n^+(t) = \sum_{k=0}^n B_{n,k}^+t^k$ and 
$B_n^-(t) = \sum_{k=0}^n B_{n,k}^-t^k$.  
Let
$\BE_{n,k}^+ = |\{ \pi \in \BB_n^+: \exc_B(\pi) = k \}|$ and 
$\BE_{n,k}^- = |\{ \pi \in \BB_n^-: \exc_B(\pi) = k \}|$.  
Define
$\BE_n^+(t) = \sum_{k=0}^n \BE_{n,k}^+t^k$ and
$\BE_n^-(t) = \sum_{k=0}^n \BE_{n,k}^-t^k$.  
Lastly define 
$\SB_n(t) = \sum_{\pi \in \BB_n} (-1)^{\inv_B(\pi)} t^{\des_B(\pi)}$ and 
$\sgnBE_n(t) = \sum_{\pi \in \BB_n} (-1)^{\inv_B(\pi)} t^{\exc_B(\pi)}$.
Reiner in \cite{reiner-descents-weyl} and 
Sivasubramanian in \cite{siva-sgn_exc_hyp}
showed the following.

\begin{theorem}[Reiner]
\label{thm:reiner-sgn-b-des}
For positive integers $n$, $\SB_n(t) = (1-t)^n$.
\end{theorem}

\begin{theorem}[Sivasubramanian]
\label{thm:siva-sgn-b-exc}
For positive integers $n$, $\sgnBE_n(t) = (1-t)^n$.
\end{theorem}

\begin{remark}
\label{rem:Brenti_siva}
Brenti in \cite[Theorem 3.15]{brenti-q-eulerian-94} 
showed that for positive integers $n$, we 
have $B_n(t) = \BE_n(t)$.
From Theorems \ref{thm:reiner-sgn-b-des} and 
\ref{thm:siva-sgn-b-exc}, we infer that 
$\SB_n(t) = \sgnBE_n(t)$ for all positive
integers $n$.  
Thus, we further have 
$B_n^+(t) = \BE_n^+(t)$ and
$B_n^-(t) = \BE_n^-(t)$. This is a noteworthy difference
between the type A and the type B Coxeter groups.
\end{remark}

The main result of this Section is the following
type B counterpart of 
Theorem \ref{thm:clt_exc_plus_minus}.  

\begin{theorem}
\label{thm:typeb-clt_exc_plus_minus}
The distribution of the coefficients of $\BE_n^{\pm}(t)$ (and hence
of $B_n^{\pm}(t)$) is asymptotically normal as $n \to \infty$.  
The random variables $X_{\exc_B}$ and $X_{\des_B}$ over $\BB_n^{\pm}$
have mean $n/2$ when $n \geq 2$, and 
have variance $(n+1)/12$ when $n \geq 3$.
\end{theorem}
\begin{proof}
For $n \geq 1$, we clearly have 
\begin{equation}
\label{eqn:factor_reln_typeB}
\BE_n^{\pm}(t) = 
\frac{1}{2} \Big( \BE_n(t) \pm \sgnBE_n(t) \Big) = 
\frac{1}{2} \Big( B_n(t) \pm (1-t)^n \Big).
\end{equation}
We have used $\BE_n(t) = B_n(t)$ above.  Thus, by Lemma 
\ref{lem:connection_between_all_and_plus}, the first $n-1$ moments of 
$\BE_n^{\pm}(t)$ are identical to the first $n-1$ moments of $B_n(t)$.  
By Theorem \ref{thm:clt-kahle-stump},
$B_n(t)$ is asymptotically normal.  Hence, by the method of moments, 
$\BE_n^{\pm}(t)$ is also asymptotically normal.  Further, they
have the same expected value as $B_n(t)$ when $n \geq 2$ and the same
variance as $B_n(t)$ when $n \geq 3$.  The proof is complete.
\end{proof}

\section{Type D Coxeter Groups}
\label{sec:typed}

The type D Coxeter group $\DD_n \subseteq \BB_n$ can be 
combinatorially defined as $\DD_n = \{ \pi \in \BB_n: \Negs(\pi) 
\mbox{ is even} \}$.  That is, $\DD_n$ consists of those signed 
permutations in $\BB_n$ with 
an even number of negative elements.
See the 
book by Bj{\"o}rner and Brenti \cite[Chapter 8]{bjorner-brenti}.

For a positive integer $n \geq 2$, define $[n]_0 = \{0,1,\ldots,n\}$.
For 
$\pi = \pi_1,\pi_2,\ldots,\pi_n \in \DD_n$, let $\pi_0 = -\pi_2$. 
Define 
$\DES_D(\pi) = \{ i \in [n-1]_0: \pi_i > \pi_{i+1} \}$ and let
$\des_D(\pi) = |\DES_D(\pi)|$.
Define
$D_{n,k} = | \{ \pi \in \DD_n: \des_D(\pi) = k \}|$.
In $\DD^{\pm}_n$, define 
$D_{n,k}^+ = |\{ \pi \in \DD_n^+: \des_D(\pi) = k \} |$ and
$D_{n,k}^- = |\{ \pi \in \DD_n^-: \des_D(\pi) = k \} |$.  Define

\begin{equation}
\label{eqn:defn_dtype_pm_des_eul}
D_n^+(t) = \sum_{k=0}^n D_{n,k}^+t^k  
\hspace{0.32 cm} \mbox{ and } \hspace{0.32 cm}
D_n^-(t) = \sum_{k=0}^n D_{n,k}^-t^k.  
\end{equation}

As $\DD_n \subseteq \BB_n$, we use the same definition of 
excedance in $\DD_n$.  Thus, for $\pi \in \DD_n$, we have 
$\exc_D(\pi) = \exc_B(\pi)$. Let 
$\DE_{n,k}^+ = |\{ \pi \in \DD_n^+: \exc_B(\pi) = k \}|$ and 
$\DE_{n,k}^- = |\{ \pi \in \DD_n^-: \exc_B(\pi) = k \}|$.  
Define
$\DE_n^+(t) = \sum_{k=0}^n \DE_{n,k}^+t^k$ and
$\DE_n^-(t) = \sum_{k=0}^n \DE_{n,k}^-t^k$.
Lastly, let  
$\SDD_n(t) = \sum_{\pi \in \DD_n} (-1)^{\inv_D(\pi)} t^{\des_D(\pi)}$ and
$\sgnDE_n(t) = \sum_{\pi \in \DD_n} (-1)^{\inv_D(\pi)} t^{\exc_D(\pi)}$.
Reiner in \cite{reiner-descents-weyl} and
Sivasubramanian in \cite{siva-weyl_linear}
showed the following.

\begin{theorem}[Reiner]
\label{thm:reiner-sgn-d-des}
For positive integers $n$,
$$
\SDD_n(t)= 
\begin{cases}
(1-t)^n & \text {if $n$ is even, }\\
(1+t)(1-t)^{n-1} & \text {if $n$ is odd.}
\end{cases}
$$
\end{theorem}

\begin{theorem}[Sivasubramanian]
\label{thm:siva-sgn-d-exc}
For positive integers $n$, 
$$
\sgnDE_n(t)= \begin{cases}
(1-t)^n & \text {if $n$ is even, }\\
(1-t)^{n-1} & \text {if $n$ is odd. }
\end{cases}
$$
\end{theorem}

The first main result of this Section is the following
type D counterpart of Theorem \ref{thm:clt_exc_plus_minus}.  
\begin{theorem}
\label{thm:typed-clt_des_plus_minus}
The distribution of the coefficients of $D_n^{\pm}(t)$ 
is asymptotically
normal as $n \to \infty$.  Over $\DD_n^{\pm}$, the random variable 
$X_{\des_D}$ has mean $n/2$ when $n \geq 3$, and has 
variance $(n+2)/12$ when $n \geq 4$.
\end{theorem}
\begin{proof}
For $n \geq 1$, we clearly have
$$
D_n^{\pm}(t)  =  \frac{1}{2} \Big( D_n(t) \pm \SDD_n(t) \Big)  = 
\begin{cases}
\displaystyle \frac{D_n(t) \pm (1-t)^n}{2} & \text {if $n$ is even, }
\vspace{2 mm}
\\ 
\displaystyle \frac{D_n(t) \pm (1+t)(1-t)^{{n-1}}}{2} & \text {if $n$ is odd. }
\end{cases}
$$
Thus, irrespective of the parity of $n$, 
by Lemma \ref{lem:connection_between_all_and_plus}, the first
$n-2$ moments of $D_n^{\pm}(t)$ are identical to the
first $n-2$ moments of $D_n(t)$.  Since the coefficients of $D_n(t)$ 
are asymptotically normal, by the method of moments, so are the 
coefficients of $D_n^{\pm}(t)$.  Further, they
have the same expected value as $D_n(t)$ when $n \geq 3$ and the 
same variance as $D_n(t)$ when $n \geq 4$.  The proof is complete.
\end{proof}

\begin{remark}
\label{rem:typed-exc-done}
In \cite[Theorem 16]{siva-dey-gamma_positive_excedances_alt_group}, Dey
and Sivasubramanian show that the sum of 
$\exc_B$ 
over  $\BB_n^+$ equals the sum of 
$\exc_D$ over  $\DD_n$.  Hence 
by Theorem \ref{thm:typeb-clt_exc_plus_minus}, 
a CLT follows for $X_{\exc_D}$ over $\DD_n$.
\end{remark}

Another result of this Section is the following
type D counterpart of Theorem \ref{thm:clt_exc_plus_minus}.  
\begin{theorem}
\label{thm:typed-clt_exc_plus_minus}
The distribution of the coefficients of $\DE_n^{\pm}(t)$ 
is asymptotically
normal as $n \to \infty$.  Over $\DD_n^{\pm}$, the random variable 
$X_{\exc_D}$ has mean $n/2$ when $n \geq 3$, and has 
variance $(n+1)/12$ when $n \geq 4$.
\end{theorem}
\begin{proof}
For $n \geq 1$, we clearly have
$$
\DE_n^{\pm}(t)  =  \frac{1}{2} \Big( \DE_n(t) \pm \sgnDE_n(t) \Big)  = 
\begin{cases}
\displaystyle \frac{\DE_n(t) \pm (1-t)^n}{2} & \text {if $n$ is even, }\\
\displaystyle \frac{\DE_n(t) \pm (1-t)^{{n-1}}}{2} & \text {if $n$ is odd. }
\end{cases}
$$
Thus, irrespective of the parity of $n$, 
by Lemma \ref{lem:connection_between_all_and_plus}, the first
$n-2$ moments of $\DE_n^{\pm}(t)$ are identical to the
first $n-2$ moments of $\DE_n(t)$.  By Remark \ref{rem:typed-exc-done}
the coefficients of $\DE_n(t)$ are asymptotically normal.  Thus, by 
the method of moments, the coefficients of 
$\DE_n^{\pm}(t)$ are also asymptotically normal.  Further, they
have the same expected value as $\DE_n(t)$ when $n \geq 3$ and the
same variance as $\DE_n(t)$ when $n \geq 4$.  The proof is complete.
\end{proof}

Note that Theorem  
\ref{thm:typed-clt_exc_plus_minus}
refines Theorem \ref{thm:typeb-clt_exc_plus_minus}.

\section{Results on Borowiec and M\l{}otkowski's variant}
\label{sec:boroweic_mlot_refine}

In \cite{boroweic-mlotkowski-new-eulerian-type-d}, 
Borowiec and M\l{}otkowski enumerated type B descents $\des_B$ 
over $\DD_n$ and $\BB_n - \DD_n$.  Let 
$\BDes_{n,k}^D = |\{ \pi \in \DD_n: \des_B(\pi) = k \}|$ and 
$\BDes_{n,k}^{B-D} = |\{ \pi \in \BB_n - \DD_n: \des_B(\pi) = k \}|$.
They considered the polynomials
\begin{equation}
\label{eqn:boroweic_mlot_original}
\BDes^D_n(t)  =  \sum_{k=0}^n  \BDes_{n,k}^D t^k
\hspace{0.32 cm} \mbox{ and } \hspace{0.32 cm}
\BDes^{B-D}_n(t)  =  \sum_{k=0}^n \BDes_{n,k}^{B-D} t^k.
\end{equation}

In \cite[Equations (24),(25)]{boroweic-mlotkowski-new-eulerian-type-d}, 
they showed the following.  

\begin{theorem}[Borowiec and M\l{}otkowski]
\label{thm:boroweic-mlot-eqns24-25}
For positive integers $n$, the two polynomials 
$\BDes^D_n(t)$ and
$\BDes^{B-D}_n(t)$ satisfy the following:
\begin{equation}
\label{eqn:copy_boroweic}
\BDes^D_n(t) = \frac{1}{2} \Big( B_n(t) + (1-t)^n \Big) 
\hspace{0.32 cm}
\mbox{ and }
\hspace{0.32 cm}
\BDes^{B-D}_n(t)= \frac{1}{2} \Big( B_n(t) - (1-t)^n \Big).
\end{equation}
\end{theorem}

Using Theorem \ref{thm:reiner-sgn-b-des}, we restate Theorem 
\ref{thm:boroweic-mlot-eqns24-25} as follows.

\begin{corollary}
\label{cor:boroweic-mlot-eqns24-25_restate}
For positive integers $n$, 
\begin{equation}
\label{eqn:copy_boroweic_restate}
\BDes^D_n(t) = B_n^+(t) 
\hspace{0.32 cm} \mbox{ and } \hspace{0.32 cm}
\BDes^{B-D}_n(t)= B_n^-(t).
\end{equation}
\end{corollary}

\begin{remark}
\label{rem:clt_boroweic_follows}
From Theorem \ref{thm:typeb-clt_exc_plus_minus}
and Remark \ref{rem:Brenti_siva}, it follows that 
the distribution of the coefficients of $\BDes_n^D(t)$ and
$\BDes_n^{B-D}(t)$ are asymptotically normal as $n \to \infty$.
\end{remark}

Mimicking the approach of Borowiec and M\l{}otkowski, suppose 
instead of type B descents, we wish to sum type B excedances
over $\DD_n$.   Similar to \eqref{eqn:boroweic_mlot_original}
we would need to define
$\BExc_{n,k}^D = |\{ \pi \in \DD_n: \exc_B(\pi) = k \}|$ and 
$\BExc_{n,k}^{B-D} = |\{ \pi \in \BB_n - \DD_n: \exc_B(\pi) = k \}|$.
Define
\begin{equation}
\label{eqn:boroweic_mlot_original_exc}
\BExc^D_n(t)  =  \sum_{k=0}^n  \BExc_{n,k}^D t^k
\hspace{0.32 cm} \mbox{ and } \hspace{0.32 cm}
\BExc^{B-D}_n(t)  =  \sum_{k=0}^n \BExc_{n,k}^{B-D} t^k.
\end{equation}

\begin{remark}
\label{rem:boro_mlot-exc_not_new}
It is simple to see that  $\BExc_n^D(t) = \DE_n(t)$ and hence we have
no new results on this ``variant".
\end{remark}

We start work in the next subsection towards proving a CLT for the 
random variable $X_{\des_B}$ over $\DD_n^{\pm}$.

\subsection{Enumerating type B descents over $\DD_n^+$ and $\DD_n^-$}
Let $\pi = \pi_1, \pi_2,\ldots, \pi_n \in \DD_n$.
The following combinatorial definition of type D inversions
is known (see Petersen's 
book \cite[Page 302]{petersen-eulerian-nos-book}): 
$\inv_D(\pi) = \inv_A(\pi) + | \{1 \leq i < j \leq n:  -\pi_i > \pi_j \}|$.
Here $\inv_A(\pi)$ is computed with respect to the usual order
on $\ZZ$.  
Let $\pi \in \BB_n$. We will need the 
following alternate definition of $\inv_B(\pi)$ (see Petersen's
book, \cite[Page 294]{petersen-eulerian-nos-book}):
$\inv_B(\pi) = \inv_A(\pi) + | \{1 \leq i < j \leq n:  -\pi_i > \pi_j \}| + 
|\Negs(\pi)|$.  

Let $\pi \in \DD_n$.  
We can also think of $\pi$ as an element of $\BB_n$.  
Since we have a combinatorial definition of $\inv_B$,
even though $\pi \in \DD_n$, $\inv_B(\pi)$ is defined.  
Similarly, we also
have $\inv_D(\pi)$ when $\pi \in \BB_n$, especially 
when $\pi \in \BB_n-\DD_n$.
From the above definitions of $\inv_B(\pi)$ and $\inv_D(\pi)$, 
the following remark follows.  This appears in 
\cite[Equation (45)]{brenti-q-eulerian-94} of Brenti and 
we will need this later.  

\begin{remark}
	\label{rem:invb_invd_ddn}
	Let $\pi \in \BB_n$.  Then, $\inv_B(\pi) = \inv_D(\pi) + |\Negs(\pi)|$.
\end{remark}

Let $\pi \in \BB_n$.  Define its number of type B ascents 
to be $\asc_B(\pi)= n - \des_B(\pi)$.
We are interested in the following two quantities:
\begin{eqnarray}
\label{eqn:type_b_des_over_Dn_defn}
\SgnBDes^D_n(s,t) &  =  & \sum_{\pi \in \DD_n} (-1)^{\inv_D \pi} s^{\asc_B \pi }t^{\des_B \pi} \\
\label{eqn:type_b_des_over_BnminusDn_defn}
\SgnBDes^{B-D}_n(s,t) & = & \sum_{\pi \in \BB_n - \DD_n} (-1)^{\inv_D \pi} s^{\asc_B \pi }t^{\des_B \pi}. 
\end{eqnarray}

Our proof involves an elaborate sign reversing involution to 
describe which we need a few preliminaries.
Let $\pi \in \DD_n$.
For an index $k \in [n]$, and positive value $r \in [n]$, define
$\pos_r(\pi)=k$ if $\pi_{k} =r$ and 
define $\pos_{-r}(\pi)=k$ if $\pi_{k} =\ob{r}$. 
Moreover, let $\pos_{\pm r}(\pi)=k$ if $\pi_k=r$ or $\pi_k = \ob{r}$.

Let $n \geq 3$. For $\pi=\pi_1,\pi_2, \dots ,\pi_n \in \DD_n$, 
let $\pi''$ be obtained from $\pi$ by deleting the letters $n$ and $n-1$.  
Suppose $\pi \in \DD_n$.  Then, it is easy 
to see that both
$\pi'' \in \DD_{n-2}$ and $\pi'' \in \BB_{n-2} - \DD_{n-2}$ 
are possible.  
From a permutation $\pi'' \in \DD_{n-2}$, to get a permutation 
$\pi \in \DD_n$, we have to insert either 
$n$ and $n-1$ or insert $\ob{n}$ and $\ob{n-1}$.
Similarly, from $\pi'' \in \BB_{n-2}-\DD_{n-2}$, to get 
$\pi \in \DD_n$, we have to insert either 
$n$ and $\ob{n-1}$ or insert $\ob{n}$ and $n-1$.
A similar statement is true  when we want to get $\pi \in \BB_n - \DD_n$ 
from $\pi'' \in \BB_{n-2} - \DD_{n-2}$ or from 
$\pi'' \in \DD_{n-2}$.  

We partition $\DD_n$ into the following $6$ disjoint subsets and will
consider the contribution of each set to $\SgnBDes_n(s,t)$.

\begin{enumerate}
	\item  $\DD_n^1= \{ \pi \in \DD_n :\pi'' \in \DD_{n-2}, 
|\pos_{\pm n}(\pi) - \pos_{\pm n-1}(\pi)|=1 , 
\pi_n \in \{ \pm (n-1), \pm n  \} \}$,
	\item $\DD_n^2= \{ \pi \in \DD_n : \pi'' \in \DD_{n-2}, 
|\pos_{n}(\pi) - \pos_{n-1}(\pi)| =1 , \pi_n \notin \{  n-1,  n  \} \} $,
	\item $\DD_n^3= \{ \pi \in \DD_n :\pi'' \in \DD_{n-2}, 
|\pos_{-n}(\pi) - \pos_{-(n-1)}(\pi)| =1 ,  \pi_n 
\notin \{ \ob{n-1}, \ob{n}  \} \}$ ,
	\item $\DD_n^4= \{ \pi \in \DD_n : \pi'' \in \DD_{n-2},  
|\pos_{n}(\pi) - \pos_{n-1}(\pi)| > 1  \}$ ,
	\item $\DD_n^5= \{ \pi \in \DD_n : \pi'' \in \DD_{n-2},  |\pos_{-n}(\pi) - \pos_{-(n-1)}(\pi)| > 1  \}$ ,
	\item  $\DD_n^6= \{ \pi \in \DD_{n} :\pi'' \in \BB_{n-2}-\DD_{n-2} \}$ .
\end{enumerate}

We prove some preliminary lemmas which we will use later.

\begin{lemma}
\label{lemma:contribution_n_n-1_together_but_not_at_last_coming_from_D}
For positive integers $n \geq 3$, the contribution 
of $\DD_n^2 \cup \DD_n^3 $ to  $\SgnBDes^D_n(s,t)$ is $0$.  That is,   
$$\sum_{\pi \in \DD_n^2 \cup \DD_n^3} 
(-1)^{\inv_D(\pi)}s^{\asc_B(\pi)}t^{\des_B(\pi)}=0$$ 
\end{lemma}

\begin{proof}
Let $\pi \in \DD_n^2$.  Suppose  $\pos_{n}(\pi) - \pos_{n-1}(\pi)=-1$. 
Then, $\pi$ has the following form: 
$\displaystyle \pi= \pi_1, \dots, \pi_{i-1},\pi_i= n, \pi_{i+1}=n-1, \pi_{i+2}, \dots, \pi_n.$ 
Define  $f: \DD_n^2 \mapsto \DD_n^3$ by 
$f(\pi)= \pi_1, \dots,\pi_{i-1},\pi_i= -(n-1), \pi_{i+1}=-n, \pi_{i+2}, \dots, \pi_n$. We have
 \begin{eqnarray} 
 \label{eqn:preservinginversion_D}
 \inv_D(f(\pi)) & = & \inv_{B}(f(\pi))-|\Negs(f(\pi))|  \nonumber \\ 
 & \equiv  & \inv_{B}( \pi_1, \dots, \pi_{i-1},\pi_i= n-1, \pi_{i+1}=n, \pi_{i+2}, \dots, \pi_n)-|\Negs(\pi) |  \>  ( \hspace{-4 mm} \mod 2) \nonumber \\
 & \equiv & \inv_B(\pi) -1 \>\>\> (\hspace{-4 mm} \mod 2)
 \equiv \inv_D(\pi) -1 \>\> (\hspace{-4 mm} \mod 2).
 \end{eqnarray}
 The second line above uses the fact that  
 flipping the sign of a single $\pi_i$ changes the parity of the 
 number of type B inversions (see \cite[Lemma 3]{siva-sgn_exc_hyp})
The last line uses the fact that swapping two letters of $\pi$ changes 
the parity of type B inversions.
Moreover, it is easy to check that $\des_B(f(\pi))=\des_B(\pi)$.

When $\pos_{n}(\pi) - \pos_{n-1}(\pi)=1$, 
$\pi$ has the form  $\pi = \pi_1, \dots, \pi_{i-1},\pi_i= n-1, \pi_{i+1}=n, \pi_{i+2}, \dots, \pi_n$.
We define $f(\pi)= \pi_1, \dots, ,\pi_{i-1},\pi_i=\ob{n}, \pi_{i+1}=\ob{n-1}, 
\pi_{i+2}, \dots, \pi_n$.
As done before,  one can check that  $\des_B(f(\pi))= \des_B(\pi)$ 
 and $\inv_D(f(\pi)) - \inv_D(\pi) \equiv  1 \mod 2$. Moreover, $f$ is invertible. The proof is complete. 
\end{proof}

\begin{lemma}
\label{lemma:contribution_n_n-1_not_together_coming_from_D}
For positive integers $n \geq 3$, the contribution of $\DD_n^4$ to  
$\SgnBDes^D_n(s,t)$ is $0$.  That is,   
$$\sum_{\pi \in \DD_n^4}
(-1)^{\inv_D(\pi)}s^{\asc_B(\pi)}t^{\des_B(\pi)}=0$$ 
\end{lemma}

\begin{proof}
Let $\pi = \pi_1, \dots,\pi_{i-1},\pi_i= n, \pi_{i+1}, \dots, 
\pi_j=n-1, \dots, \pi_n \in \DD_n^4$. Define $g: \DD_n^4 \mapsto \DD_n^4$ 
by $g(\pi)= \pi_1, \dots, ,\pi_{i-1},\pi_i= n-1, \pi_{i+1}, \dots, 
\pi_j=n, \dots, \pi_n$.  The map $g$ clearly satisfies 
$\des_B(\pi) = \des_B(g(\pi))$ but 
changes the parity of $\inv_D$ as the pair $(i,j)$ flips being
an inversion. The proof is complete. 
\end{proof}

\begin{lemma}
\label{lemma:contribution_minusn_minusn-1_not_together_coming_from_D}
For positive integers $n \geq 3$, the contribution of $\DD_n^5$ to  
$\SgnBDes^D_n(s,t)$ is $0$.   That is,   
$$\sum_{\pi \in \DD_n^5}(-1)^{\inv_D(\pi)}s^{\asc_B(\pi)}t^{\des_B(\pi)}=0.$$ 
\end{lemma}

The proof of Lemma 
\ref{lemma:contribution_minusn_minusn-1_not_together_coming_from_D} 
is very similar to the proof of Lemma 
\ref{lemma:contribution_n_n-1_not_together_coming_from_D} and so we 
omit it.

\begin{lemma}
\label{lemma:contribution_BMINUSD}
For positive integers $n \geq 3$, the contribution of $\DD_n^6$ to  
$\SgnBDes^D_n(s,t)$ is $0$.  That is,   
$$\sum_{\pi \in \DD_n^6}(-1)^{\inv_D(\pi)}s^{\asc_B(\pi)}t^{\des_B(\pi)}=0$$ 
\end{lemma}

\begin{proof}
Let $\pi \in \DD_n^6 $ and $\pi '' \in \BB_{n-2}- \DD_{n-2}$. Thus, the 
one line notation of $\pi$ either contains 
$n$ and $\ob{n-1}$ or contains $\ob{n}$ and $n-1$.  Firstly, 
suppose 
$$\pi = \pi_1, \dots,\pi_{i-1},\pi_i= n, \pi_{i+1}, \dots, 
\pi_j=\ob{n-1}, \dots, \pi_n .$$ 
Define $h: \DD_n^6 \mapsto \DD_n^6$ by 
$h(\pi)= \pi_1, \dots,\pi_{i-1},\pi_i= n-1, \pi_{i+1}, \dots, 
\pi_j=\ob{n}, \dots, \pi_n$.
It is easy to check that $h$ preserves $\des_B$ but changes the 
parity of $\inv_D$. A very similar map can be given if $\pi$ contains
$\ob{n}$ and $n-1$, completing the  proof.
\end{proof}

Thus, the total contribution of the sets $\DD_n^k$ for $k \geq 2$,
to $\SgnBDes^D_n(s,t)$, equals $0$.  Hence
\begin{equation}
\label{eqn:D_equals_D1}
\SgnBDes^D_n(s,t) =   \sum_{\pi \in \DD_n^1} (-1)^{\inv_D \pi} s^{\asc_B \pi }t^{\des_B \pi}. 
\end{equation}

\begin{theorem}
\label{thm:main_theorem_signed_desB_enumerate_over_type_d}
For positive integers $n \geq 2$, the following recurrence relations hold: 
\begin{eqnarray}
\label{eqn:recurrence_by_jump_2_D}
\SgnBDes^{D}_n(s,t) & = & (s-t)^2 \SgnBDes^D_{n-2}(s,t), \\
\label{eqn:recurrence_by_jump_2_B-D}
\SgnBDes^{B-D}_n(s,t) & = &  (s-t)^2 \SgnBDes^{B-D}_{n-2}(s,t).
\end{eqnarray}
We thus have 
\begin{eqnarray}
\label{eqn:first_imp}
\SgnBDes^{D}_n(s,t)& = &\begin{cases}
(s-t)^n & \text {when $n$ is even },\\
s(s-t)^{n-1} & \text {when $n$ is odd }.
\end{cases}   \\
\label{eqn:second_imp}
\SgnBDes^{B-D}_n(s,t) & =& \begin{cases}
0 & \text {when $n$ is even },\\
t(s-t)^{n-1} & \text {when $n$ is odd }.
\end{cases}
\end{eqnarray}

\end{theorem}

\begin{proof}
We consider \eqref{eqn:recurrence_by_jump_2_D} first.
Each $\pi=\pi_1,\pi_2,\dots,\pi_{n-2} \in \DD_{n-2}$ gives 
rise to the following four permutations
$\psi_1$, $\psi_2$, $\psi_3$ and $\psi_4$ in $\DD_n^1$: 
\begin{eqnarray*}
\psi_1 = \pi_1,\pi_2,\dots, \pi_{n-2},n-1,n, \hspace{2 cm}  & & 
\psi_2 = \pi_1,\pi_2,\dots, \pi_{n-2},n,n-1, \\
\psi_3 = \pi_1,\pi_2,\dots, \pi_{n-2},\ob{n},\ob{n-1}, \hspace{2 cm} & & 
\psi_4 = \pi_1,\pi_2,\dots, \pi_{n-2},\ob{n-1},\ob{n}.
\end{eqnarray*}

It is simple to note that 

\begin{enumerate}
	\item $\des_B(\psi_1 )= \des_B(\pi)$ and $\inv_D(\psi_1) - \inv_D(\pi) \equiv 0 \mod 2$.
	
	\item $\des_B(\psi_i )= \des_B(\pi)+1$ and 
	$\inv_D(\psi_i) - \inv_D(\pi) \equiv 1 \mod 2$ when $i=2,3$.
	
	
	\item $\des_B(\psi_4 )= \des_B(\pi)+2$ and 
	$\inv_D(\psi_4) - \inv_D(\pi) \equiv 0 \mod 2$.
\end{enumerate}

Thus, we get
\begin{eqnarray*}
\SgnBDes^D_n(s,t) & = &	 
\sum_{\pi \in \DD_n^1} (-1)^{\inv_D \pi} s^{\asc_B \pi }t^{\des_B \pi} \\
& = & (s^2-2st+t^2) \sum_{\pi \in \DD_{n-2}} (-1)^{\inv_D \pi} s^{\asc_B \pi }t^{\des_B \pi} 
 =  (s-t)^2 \SgnBDes^D_{n-2}(s,t)
\end{eqnarray*}

Variants of Lemmas 
\ref{lemma:contribution_n_n-1_together_but_not_at_last_coming_from_D},
\ref{lemma:contribution_n_n-1_not_together_coming_from_D},
\ref{lemma:contribution_minusn_minusn-1_not_together_coming_from_D} and
\ref{lemma:contribution_BMINUSD}
can be proved for $\BB_n - \DD_n$ by defining $(\BB_n- \DD_n)^k$ for
$1 \leq k \leq 6$.
Using these, in a similar manner, one can 
prove \eqref{eqn:recurrence_by_jump_2_B-D}.
One can check the following base cases: 
$\SgnBDes^D_1(s,t)=s$,  $\SgnBDes^D_2(s,t)=(s-t)^2$ and 
$\SgnBDes^{B-D}_1(s,t)=t$, $\SgnBDes^{B-D}_2(s,t)=0$.  
Using these with 
\eqref{eqn:recurrence_by_jump_2_D} and \eqref{eqn:recurrence_by_jump_2_B-D}
completes the proof. 
\end{proof}

We mention two uses of 
Theorem \ref{thm:main_theorem_signed_desB_enumerate_over_type_d} before
moving on to the proof of CLTs.
Firstly, subtracting \eqref{eqn:second_imp} from 
\eqref{eqn:first_imp}, we get an alternative proof of a 
bivariate version of Theorem \ref{thm:reiner-sgn-b-des}.

\begin{theorem}[Reiner]
\label{thm:reiner}
For positive integers $n$, $\sum_{\pi \in \BB_n} (-1) ^{\inv_B(\pi)} s^{\asc_B(\pi)}t^{\des_B(\pi)}= (s-t)^n.$ 
\end{theorem}

Next, as mentioned in Remark \ref{rem:Brenti_siva}, we have 
$B_n^+(t) = \BE^+_n(t)$ and $B_n^-(t) = \BE^-_n(t)$.  Using
Theorem \ref{thm:main_theorem_signed_desB_enumerate_over_type_d},
we get the following further refinement which we record below.

\begin{remark}
\label{rem:excb_desb_equidistrib_over_several_sets}
From \cite[Remark 26]{siva-dey-gamma_positive_excedances_alt_group}
of Dey and Sivasubramanian,
we see that $\des_B$ and $\exc_B$ are equidistributed over 
$\DD_n$ and hence over $\BB_n - \DD_n$.
Theorem \ref{thm:siva-sgn-d-exc} and \eqref{eqn:first_imp} of 
Theorem 
\ref{thm:main_theorem_signed_desB_enumerate_over_type_d} show
that $\sgnDE_n(t) = \SgnBDes^{D}_n(1,t)$.  Thus, we get
$\DE_n^+(t) = \SgnBDes^{D,+}_n(1,t)$ and 
$\DE_n^-(t) = \SgnBDes^{D,-}_n(1,t)$.  That is, 
$\des_B$ and $\exc_B$ are equidistributed over $\DD_n^{\pm}$.
One can show in a similar manner that
$\des_B$ and $\exc_B$ are equidistributed
over $(\BB_n - \DD_n)^{\pm}$.
\end{remark}

\begin{remark}
Borowiec and M\l{}otkowski in 
\cite[Proposition 4.7]{boroweic-mlotkowski-new-eulerian-type-d}
show another signed enumeration result which is similar to Theorem
\ref{thm:main_theorem_signed_desB_enumerate_over_type_d},
but is different from it.
\end{remark}

\subsection{Fixed points of the involution in the proof of Theorem 
\ref{thm:main_theorem_signed_desB_enumerate_over_type_d}}
It is clear that our proof of Theorem
\ref{thm:main_theorem_signed_desB_enumerate_over_type_d} is a 
sign reversing involution, though it is described in several parts.
The set of fixed points of this involution is thus a natural
question which we consider next.

We claim that the following two families of sets 
$L_n \subseteq \DD_n $ and $M_n \subseteq\BB_n - \DD_n$ 
are those which survives the cancellations in our proof. 
We first define these sets inductively.
Let $L_1= \DD_1$, $L_2= \DD_2$,  $M_1=\BB_1-\DD_1$, and for 
even positive integers $n$, let $M_{n}= \emptyset$.
We define the sets $L_n$ for all positive integers $n\geq 3$ 
and $M_n$ for odd integers $n\geq 3$. 
Consider $\pi=\pi_1, \pi_2, \dots, \pi_n \in L_{n-2}$.  
Using $\pi$, we form four  signed permutations 
$\tau_1$, $\tau_2$, $\tau_3$ and $\tau_4$  $\in L_n$ as follows: 
\begin{eqnarray*}
\tau_1 &= &\pi_1,\pi_2,\dots, \pi_{n-2},n-1,n, 
	\hspace{2 cm}	\tau_2 = \pi_1,\pi_2,\dots, \pi_{n-2},n,n-1, \\
\tau_3 &= &\pi_1,\pi_2,\dots, \pi_{n-2},\ob{n},\ob{n-1}, 
	\hspace{2 cm} \tau_4 = \pi_1,\pi_2,\dots, \pi_{n-2},\ob{n-1},\ob{n}.
\end{eqnarray*}

We do the same construction to get $M_n$ from $M_{n-2}$ for odd $n$.
Clearly, $L_n \subseteq \DD_n^1$  for positive integers $n$ and 
$M_n \subseteq (\BB_n - \DD_n)^1$ for odd positive integers $n$.
Further, it is clear that
$|L_{n}| = 2^n$ when $n$ is even and $|L_n|=2^{n-1} $ when $n$ is odd.
Further, is it simple to see that the elements of $L_n$ have the following 
property.  
Recall that for $\pi \in \DD_n$, $\pi''$ is obtained by deleting 
the two highest elements in absolute value.
For $\pi \in L_n$, let
$\pi_1 = \pi''$.  Then $\pi_1 \in L_{n-2}$.  Next, let  $\pi_2=\pi_1''$.  
Then $\pi_2 \in L_{n-4}$ and so on till $\pi_{\floor{n/2}}$.

We want to show that $L_n$ is the set of permutations that survive 
cancellations.  We do this when $n$ is even.  The argument is very
similar when $n$ is odd.
For even $n$, let $\pi \in \DD_n^1$.  Denote $\pi_{n/2} = \pi$.
Let $\pi_{n/2-1} = \pi_{n/2}''$, $\pi_{n/2-2} = \pi_{n/2-1}''$ and so on.
Consider the smallest positive integer $r$ such that either
$\pi_{n/2-r} \not \in L_{n-2r}$.   

As mentioned above, it is clear that if there is no such $r$, 
then $\pi \in L_n$.
Suppose such an $r$ exists.  Suppose $\pi_{n/2-r} \not \in L_{n-2r}$. 
The contribution of $\pi_{n/2-r}$ will 
cancel with the contribution of some other $\psi \in \DD_{n/2-r}^p$ 
for some $p$.  Padding $\psi$  up
with the same sequence of $2r$ elements that we removed 
from $\pi_{n/2}$ gives an element $\sigma \in \DD_n^1$ which cancels with
$\pi$.  Two points need to be checked and both are easy.  The first
point is to see why $\sigma \in \DD_n^1$. This is straightforward 
from the definition.  The second point is to see that $\pi$ and $\sigma$ 
have opposite signs but have the same $\des_B$ value.  This is also
easy to see.

Finally, it is straightforward to check that
$\sum_{\pi \in L_n} (-1)^{\inv_D \pi} s^{\asc_B \pi }t^{\des_B \pi} = 
(s-t)^n$ when $n$ is even and that $\sum_{\pi \in L_n} (-1)^{\inv_D \pi} 
s^{\asc_B \pi }t^{\des_B \pi}= s(s-t)^{n-1}$ when $n$ is odd.
Similarly, it is not hard to show when $n$ is odd, that, 
$\sum_{\pi \in M_n} (-1)^{\inv_D \pi} s^{\asc_B \pi }t^{\des_B \pi}
= t(s-t)^{n-1}$.
When $n$ is even, clearly 
$\sum_{\pi \in M_n} (-1)^{\inv_D \pi} s^{\asc_B \pi }t^{\des_B \pi}= 0$.

\subsection{CLT results over $\DD_n^{\pm}$ and $(\BB_n-\DD_n)^{\pm}$}
\label{subsec:clt-boroweic-variant}
In this subsection, we give our CLT results 
for the random variable $X_{\des_B}$ over $\DD_n^{\pm}$ and
$(\BB_n-\DD_n)^{\pm}$.
Let 
$\BDes_{n,k}^{D,+} = |\{ \pi \in \DD_n^+: \des_B(\pi) = k \}|$ and 
$\BDes_{n,k}^{D,-} = |\{ \pi \in \DD_n^-: \des_B(\pi) = k \}|$.  Further,
let 
$\BDes_{n,k}^{B-D,+} = |\{ \pi \in (\BB_n - \DD_n)^+: \des_B(\pi) = k \}|$ and
$\BDes_{n,k}^{B-D,-} = |\{ \pi \in (\BB_n - \DD_n)^-: \des_B(\pi) = k \}|$.
Let

\begin{eqnarray}
\label{eqn:type_b_des_over_Dn_pm_defn}
\BDes^{D,+}_n(t) &  = &  \sum_{k=0}^n  \BDes_{n,k}^{D,+} t^k
\hspace{8 mm} \mbox{ and } \hspace{3 mm}
\BDes^{D,-}_n(t)  =  \sum_{k=0}^n  \BDes_{n,k}^{D,-} t^k, \\
\label{eqn:type_b_des_over_BnminusDn_pm_defn}
\BDes^{B-D,+}_n(t)  & =  & \sum_{k=0}^n \BDes_{n,k}^{B-D,+} t^k
\hspace{3 mm} \mbox{ and } \hspace{3 mm}
\BDes^{B-D,-}_n(t)  =  \sum_{k=0}^n \BDes_{n,k}^{B-D,-} t^k.
\end{eqnarray}
We are now ready to prove our CLT results.  

\begin{theorem}
\label{thm:typed-clt_bdes_plus_minus}
The distribution of the coefficients of $\BDes_n^{D,\pm}(t)$ 
and $\BDes_n^{B-D,\pm}(t)$ are asymptotically
normal as $n \to \infty$.  Over $\DD_n^{\pm}$, the random variable
$X_{\des_B}$ has mean $n/2$ when $n \geq 3$, and has
variance $(n+1)/12$ when $n \geq 4$.  Over $(\BB_n - \DD_n)^{\pm}$,
the random variable $X_{\des_B}$ has mean $n/2$ when $n \geq 3$ 
and variance $(n+1)/12$ when $n \geq 4$.  
\end{theorem}
\begin{proof}
We first consider $X_{\des_B}$ over $\DD_n^{\pm}$. Let $n \geq 1$. 
Using Theorem \ref{thm:boroweic-mlot-eqns24-25} and Theorem 
\ref{thm:main_theorem_signed_desB_enumerate_over_type_d}, we get  

\begin{eqnarray}
\BDes_n^{D,\pm}(t)  & =  & 
\frac{1}{2} \Big( \BDes_n^D(t) \pm \SgnBDes^D_n(1,t) \Big)  = 
\frac{1}{2} \Big( B_n^+(t) \pm \SgnBDes^D_n(1,t) \Big)  
\nonumber \\
& =& 
\label{eqn:Bdes_pm_expression}
\begin{cases}
\frac{1}{2}\Bigg( B_n^+(t) \pm (1-t)^n \Bigg) & \text {if $n$ is even, }\\
\frac{1}{2}\Bigg( B_n^+(t) \pm (1-t)^{n-1} \Bigg) & \text {if $n$ is odd. }
\end{cases} 
\end{eqnarray}

Similarly, we have
\begin{eqnarray}
\BDes_n^{B-D,\pm}(t)  & =  & 
\frac{1}{2} \Big( \BDes_n^{B-D}(t) \pm \SgnBDes^{B-D}_n(1,t) \Big)  =
\frac{1}{2} \Big( B_n^-(t) \pm \SgnBDes^{B-D}_n(1,t) \Big)  
\nonumber \\
& =& 
\label{eqn:B-Ddes_pm_expression}
\begin{cases}
\frac{1}{2} B_n^-(t)   & \text {if $n$ is even, }\\
\frac{1}{2} \Bigg( B_n^-(t) \pm t(1-t)^{n-1} \Bigg) & \text {if $n$ is odd. }
\end{cases} 
\end{eqnarray}

Thus, irrespective of the parity of $n$,
by Lemma \ref{lem:connection_between_all_and_plus}, the first
$n-2$ moments of $\BDes_n^{D,\pm}(t)$ are identical to the
first $n-2$ moments of $B_n^+(t)$.  Since the coefficients of 
$B_n^+(t)$ are asymptotically normal, by the method of moments, 
so are the coefficients of $\BDes_n^{D,\pm}(t)$.  Further, they
have the same expected value as $B_n^+(t)$ when $n \geq 3$ 
and the same variance as $B_n^+(t)$ when $n \geq 4$.  
The proof for a CLT for $X_{\des_B}$ over $(\BB_n - \DD_n)^{\pm}$ 
is very similar and is hence omitted.
\end{proof}

\begin{remark}
By Remark \ref{rem:excb_desb_equidistrib_over_several_sets},
similar to Theorem \ref{thm:typed-clt_bdes_plus_minus},
one can get CLTs for $X_{\exc_B}$ over $\DD_n^{\pm}$ and 
$(\BB_n - \DD_n)^{\pm}$.
\end{remark}

\section{Carlitz Identities involving even and odd Excedance based 
Eulerian Polynomials}		
\label{sec:powerseriesidentities}	

In this Section, we prove Carlitz type identities which refine 
various known identities associated with Eulerian polynomials arising
from Coxeter groups.  Our identities are obtained using signed
enumeration results that either exists in the literature or
those proved in Section \ref{sec:boroweic_mlot_refine}.  We begin
with our proof of Theorem \ref{thm:carlitz-identity_exc_typeA}.

\begin{proof} (Of Theorem \ref{thm:carlitz-identity_exc_typeA})
By Theorem \ref{thm:carlitz-identity} and Theorem \ref{thm:mantaci}, 
we have 
\begin{eqnarray*}
\label{eqn:powerseriesidentitySnidentity2}
\frac{\AAE_n(t)}{(1-t)^{n+1}}  & = &   \sum_{k \geq 0}(k+1)^nt^k  
\hspace{0.5 cm} \mbox{ and } \hspace{0.5 cm}
\frac{\SgnAAE_n(t)}{(1-t)^{n+1}}  =  \sum_{k \geq 0}(k+1)t^k. \mbox{ Thus, } \\
\frac{\AAE_n^{\pm}(t)}{(1-t)^{n+1}}  & = & \sum_{k \geq 0}
\frac{1}{2} \Big( (k+1)^n \pm (k+1) \Big)t^k.
\end{eqnarray*}
The 
proof is complete.
\end{proof}

Theorem \ref{thm:carlitz-identity_des_typeA} and 
Theorem \ref{thm:carlitz-identity_exc_typeA} are 
two different refinements of Theorem 
\ref{thm:carlitz-identity}.  This stems from the fact 
that when $n \geq 3$, $A_n^+(t) \not= \AAE_n^+(t)$.

\subsection{Type B Carlitz identities involving $\BE_n^{\pm}(t)$}
\label{subsec:typeb_worp}

Type B analogues of the Carlitz identity are also known.  
We start with the following
result of Brenti \cite[Theorem 3.4]{brenti-q-eulerian-94}.

\begin{theorem}[Brenti]
\label{thm:carlitz-identity_typeB}
Let $B_n(t)$ be the type B Eulerian polynomial defined in
\eqref{eqn:Eulerian_poly_B_and_Dn_defn}. Then,
\begin{equation}
\label{carlitz-typeb}
\frac{B_n(t)}{(1-t)^{n+1}} = \sum_{k \geq 0} (2k+1)^nt^k.
\end{equation}
\end{theorem}

By Remark \ref{rem:Brenti_siva}, for positive integers $n$,
we have $B_n^+(t) = \BE_n^+(t)$.  Thus, we have only 
one refinement of Theorem \ref{thm:carlitz-identity_typeB}.

\begin{theorem}
\label{thm:carlitz-identity_pm_des_typeB}
Let $B_n^{\pm}(t)$ be the signed type B Eulerian polynomial
defined in
\eqref{eqn:Eulerian_poly_Bn-pm_defn}. Then,
\begin{equation}
\label{carlitz-typeb_pm}
\frac{B_n^{\pm}(t)}{(1-t)^{n+1}} = 
\frac{\BE_n^{\pm}(t)}{(1-t)^{n+1}} = 
\sum_{k \geq 0} \frac{(2k+1)^n \pm 1}{2} t^k.
\end{equation}
\end{theorem}
\begin{proof}
By Theorem \ref{thm:carlitz-identity_typeB} and Theorem 
\ref{thm:siva-sgn-b-exc}, we have 
\begin{eqnarray*}
\label{eqn:powerseriesidentityBn}
\frac{\BE_n(t)}{(1-t)^{n+1}}  & = &   \sum_{k \geq 0}(2k+1)^nt^k  
\hspace{0.51 cm} \mbox{ and } \hspace{0.51 cm}
\frac{\sgnBE_n(t)}{(1-t)^{n+1}}  =  \sum_{k \geq 0}t^k. 
\hspace{0.5 cm} \mbox{Thus,}  \\
\frac{\BE_n^{\pm}(t)}{(1-t)^{n+1}}  & = & \sum_{k \geq 0}
\frac{1}{2} \Big( (2k+1)^n \pm 1 \Big)t^k.
\end{eqnarray*}
The proof is complete.
\end{proof}

\subsection{Type D Carlitz identities involving $D_n^{\pm}(t)$}
\label{subsec:typed_worp}

Brenti in \cite[Corollary 4.8]{brenti-q-eulerian-94} gave 
the following recurrence involving the polynomials $B_n(t)$,
$D_n(t)$ and $A_{n-1}(t)$.  He also mentions that Stembridge
proves it in
\cite[Lemma 9.1]{stembridge-perm-rep-weyl-grp-cohomology-toric-variety}.

\begin{lemma}[Brenti]
\label{lem:brenti_typed_relation}
For positive integers $n$, the numbers $A_{n-1,k}, B_{n,k}$ and 
$D_{n,k}$ satisfy the following relation:
\begin{equation}
\label{rec:recurrence_relation_abd}
B_{n,k}  =  D_{n,k} + n2^{n-1} A_{n-1,k-1}.
\end{equation}
Equivalently, the Eulerian polynomials of types A, B 
and D are related as follows:
\begin{equation}
\label{eqn:eulerian_relation_brenti}
B_n(t)  =  D_n(t) + n 2^{n-1} t A_{n-1}(t).
\end{equation}
\end{lemma}

Using Lemma \ref{lem:brenti_typed_relation}, Brenti in 
\cite[Theorem 4.10]{brenti-q-eulerian-94} gave the following 
type D analogue of the Carlitz identity.
Let $\BerB_n(x)$ be the $n$-th Bernoulli polynomial.

\begin{theorem}[Brenti]
\label{thm:carlitz-identity_typeD}
For positive integers $n \geq 2$,
\begin{equation}
\label{carlitz-typed}
\frac{D_n(t)}{(1-t)^{n+1}} = 
\sum_{k \geq 0} \Big( \big(2k+1\big)^n 
-2^{n-1}\big( \BerB_n(k+1) - \BerB_n(k) \big) \Big)t^k.
\end{equation}
\end{theorem}

Our refinement involving $D_n^{\pm}(t)$ 
is the following.  The proof
uses Theorem \ref{thm:carlitz-identity_typeD} and 
Theorem \ref{thm:reiner-sgn-d-des} in a manner very
similar to the proof of Theorem \ref{thm:carlitz-identity_pm_des_typeB}.
Thus, we omit its proof and merely state out result.

\begin{theorem}
\label{thm:carlitz-identity_pm_des_typeD}
Let $D_n^{\pm}(t)$ be the signed type D Eulerian polynomial
defined in \eqref{eqn:defn_dtype_pm_des_eul}.  Then, for 
positive integers $n \geq 2$,

\begin{equation}
\label{carlitz-typed_pm}
\frac{D_n^{\pm}(t)}{(1-t)^{n+1}} = 
\begin{cases}
\sum_{k \geq 0} \frac{1}{2} \Big( \big(2k+1\big)^n 
-2^{n-1}\big( \BerB_n(k+1) - \BerB_n(k) \big)  \pm 1 \Big)t^k 
& \text {if $n$ is even, }\\
\sum_{k \geq 0} \frac{1}{2}\Big( \big(2k+1\big)^n 
-2^{n-1}\big( \BerB_n(k+1) - \BerB_n(k) \big) \pm (2k+1)\Big)t^k 
& \text {if $n$ is odd. }
\end{cases}
\end{equation}
\end{theorem}

Lemma \ref{lem:brenti_typed_relation} and Theorem 
\ref{thm:carlitz-identity_typeD} are intimately connected.
Similar to their relation, we get the following Lemma
from Theorem \ref{thm:carlitz-identity_pm_des_typeD}.
It is clear Lemma \ref{lem:stembridge-refine} refines 
Lemma \ref{lem:brenti_typed_relation}.

\begin{lemma}
\label{lem:stembridge-refine}
For positive integers $n$, the Eulerian numbers $A_{n-1,k}, 
B_{n,k}^{\pm}$ and $D_{n,k}^{\pm}$ satisfy the following relation:
\begin{equation}
\label{rec:recurrence_relation_abd_pm}
B_{n,k}^{\pm}  =  
\begin{cases}
D_{n,k}^{\pm} + \frac{1}{2} n2^{n-1} A_{n-1,k-1} & \text {if $n$ is even, }\\
D_{n,k}^{\pm} + \frac{1}{2} n2^{n-1} A_{n-1,k-1} \mp (-1)^{k-1} \binom{n-1}{k-1} & \text {if $n$ is odd.}
\end{cases}
\end{equation}
Equivalently, the signed Eulerian polnomials of types A, B 
and D are related as follows:
\begin{equation}
\label{eqn:eulerian_relation_brenti_pm}
B_n^{\pm}(t)  =  
\begin{cases}
D_n^{\pm}(t) + \frac{1}{2}n 2^{n-1} t A_{n-1}(t), & \text {if $n$ is even, }\\
D_n^{\pm}(t) + \frac{1}{2}n 2^{n-1} t A_{n-1}(t) \mp t (1-t)^{n-1}, & \text {if $n$ is odd.}
\end{cases}
\end{equation}
\end{lemma}

\begin{proof}(Sketch)
Since \eqref{rec:recurrence_relation_abd_pm} and
\eqref{eqn:eulerian_relation_brenti_pm} are equivalent, we only sketch
the proof of \eqref{rec:recurrence_relation_abd_pm}.  Clearly,
$B_n^{\pm}(t) = \frac{1}{2} \left( B_n(t) \pm \SB_n(t) \right)$
and
$D_n^{\pm}(t) = \frac{1}{2} \left( D_n(t) \pm \SDD_n(t) \right)$.
The proof follows by combining the above two with 
Theorem \ref{thm:reiner-sgn-b-des}, 
Theorem \ref{thm:reiner-sgn-d-des} and
\eqref{eqn:eulerian_relation_brenti}. 
\end{proof}

\subsection{Identities involving 
$\BDes^{D,\pm}_n(t)$ and $\BDes^{B-D,\pm}_n(t)$}
We wish to prove similar identities involving $\BDes^{D,\pm}_n(t)$.
Combining Theorem \ref{thm:boroweic-mlot-eqns24-25} with Theorem 
\ref{thm:main_theorem_signed_desB_enumerate_over_type_d}, we get
the following Carlitz type result.  

\begin{theorem}
\label{thm:carlitz-identity_pm_boroweic_typeD}
For positive integers $n$, 
\begin{equation}
\frac{\BDes_n^{D,\pm}(t)}{(1-t)^{n+1}}   = 
\begin{cases}
\sum_{k \geq 0} 
\frac{1}{2} \Big( \frac{1}{2}\big( (2k+1)^n+1 \big) \pm 1 \Big)t^k 
& \text {if $n$ is even, }\\
\sum_{k \geq 0} 
\frac{1}{2} \Big( \frac{1}{2}\big( (2k+1)^n+1 \big) \pm (k+1) \Big)t^k 
& \text {if $n$ is odd. }
\end{cases} 
\end{equation}

\begin{equation}
\frac{\BDes_n^{B-D,\pm}(t)}{(1-t)^{n+1}}   = 
\begin{cases}
\sum_{k \geq 0} 
\frac{1}{2} \Big( \frac{1}{2}\big( (2k+1)^n-1 \big) \Big)t^k 
& \text {if $n$ is even, }\\
\sum_{k \geq 0} 
\frac{1}{2} \Big( \frac{1}{2}\big( (2k+1)^n-1 \big) \pm k \Big)t^k 
& \text {if $n$ is odd. }
\end{cases} 
\end{equation}
\end{theorem}

\begin{proof}
We clearly have 
\begin{eqnarray*}
\frac{B_n^{\pm}(t)}{(1-t)^{n+1}}  & = &   \sum_{k \geq 0}\frac{1}{2}\Big(
(2k+1)^n \pm 1 \Big)t^k \\
\frac{(1-t)^{n-1}}{(1-t)^{n+1}}  & =  & \sum_{k \geq 0}(k+1)t^k  
\hspace{5 mm}  \mbox{ and } \hspace{5 mm}
\frac{(1-t)^n}{(1-t)^{n+1}}   =  \sum_{k \geq 0} t^k
\end{eqnarray*}

Combining the above three with \eqref{eqn:Bdes_pm_expression} and 
\eqref{eqn:B-Ddes_pm_expression} completes the proof.
\end{proof}

Theorem \ref{thm:carlitz-identity_pm_boroweic_typeD} clearly 
refines the Worpitzky type identity 
\cite[Proposition 4.6]{boroweic-mlotkowski-new-eulerian-type-d}
proved by Borowiec and M\l{}otkowski.

\comment{
\subsubsection{Identities involving $\BE^{B-D,\pm}_n(t)$}
We again merely state  our final result and omit its proof.

\begin{theorem}
\label{thm:carlitz-identity_pm_boroweic_typeD_excB}
For positive integers $n$, 
\begin{equation}
\frac{\BE_n^{B-D,\pm}(t)}{(1-t)^{n+1}}   = 
\begin{cases}
\sum_{k \geq 0} 
\frac{1}{2} \Big( \frac{1}{2}\big[ (2k+1)^n-1 \big] \Big)t^k 
& \text {if $n$ is even, }\\
\sum_{k \geq 0} 
\frac{1}{2} \Big( \frac{1}{2}\big[ (2k+1)^n-1 \big] \mp k^2 \Big)t^k 
& \text {if $n$ is odd. }
\end{cases} 
\end{equation}
\end{theorem}
}

\vspace{2 mm}

In $\SSS_n$ and $\BB_n$, bivariate versions of Carlitz's identity are 
known with respect to the major index and descent statistics 
see 
\cite{carlitz-amm-maj-des-carlitz-identity,
adin-brenti-roichman,chow-gessel-des-major-hyperoctahedral-group}.  
It would be interesting to see if any counterparts can
be given for $\SSS_n^{\pm}$ and $\BB_n^{\pm}$.

The preprint \cite{worpitzk-identity-signed-even-signed} by Bagno, Garber
and Novick gives combinatorial proofs of Worpitzky's identities, which 
are a rephrasement  of Carlitz's identities.  It is easy to recast all
Carlitz type results of this Section as Worpitzky identities.  
On the lines of \cite{worpitzk-identity-signed-even-signed}, it would be 
interesting to see if combinatorial proofs of our Worpitzky's
identities can be given.

\section*{Acknowledgement}
The first author acknowledges support from a CSIR-SPM
Fellowship.

The second author thanks Professor Alladi Subramanyam for illuminating
discussions on CLTs.  He acknowledges support from project grant 
P07 IR052, given by IRCC, IIT Bombay and from project
SERB/F/252/2019-2020 given by the Science and Engineering 
Research Board (SERB), India.

\bibliographystyle{acm}
\bibliography{main}

\begin{thebibliography}{10}

\bibitem{adin-brenti-roichman}
{\sc Adin, R.~M., Brenti, F., and Roichman, Y.}
\newblock Descent numbers and major indices for the hyperoctahedral group.
\newblock {\em Advances in Applied Math 27\/} (2001), 210--224.

\bibitem{worpitzk-identity-signed-even-signed}
{\sc Bagno, E., Garber, D., and Novick, M.}
\newblock The worpitzky identity for the groups of signed and even-signed
  permutations.
\newblock {\em https://arxiv.org/abs/2004.03681 arxiv preprint\/} (2019).

\bibitem{bender-central-local-limit}
{\sc Bender, E.~A.}
\newblock Central and local limit theorems applied to asymptotic enumeration.
\newblock {\em Journal of Combinatorial Theory, series A 15\/} (1973), 91--111.

\bibitem{bjorner-brenti}
{\sc Bj{\"o}rner, A., and Brenti, F.}
\newblock {\em Combinatorics of Coxeter Groups}.
\newblock GTM 231, Springer Verlag, 2005.

\bibitem{boroweic-mlotkowski-new-eulerian-type-d}
{\sc Borowiec, A., and M\l{}otkowski, W.}
\newblock New {E}ulerian numbers of type {D}.
\newblock {\em Electronic Journal of Combinatorics 23(1)\/} (2016), \#P 1.38.

\bibitem{brenti-q-eulerian-94}
{\sc Brenti, F.}
\newblock $q$-{E}ulerian {P}olynomials {A}rising from {C}oxeter {G}roups.
\newblock {\em European Journal of Combinatorics 15\/} (1994), 417--441.

\bibitem{canfield-asymptotic_normality_enumeration}
{\sc Canfield, R.~E.}
\newblock Asymptotic {N}ormality in {E}numeration.
\newblock In {\em Handbook of Enumerative Combinatorics}, M.~Bona, Ed. Chapman
  \& Hall \/ CRC Press, 2015, ch.~3.

\bibitem{carlitz-amm-maj-des-carlitz-identity}
{\sc Carlitz, L.}
\newblock A combinatorial property of $q$-{E}ulerian numbers.
\newblock {\em American Math Monthly 82\/} (1975), 51--54.

\bibitem{chow-gessel-des-major-hyperoctahedral-group}
{\sc Chow, C.-O., and Gessel, I.~M.}
\newblock On the descent numbers and major indices for the hyperoctahedral
  group.
\newblock {\em Advances in Applied Mathematics 38\/} (2007), 275--301.

\bibitem{david-barton-combin-chance}
{\sc David, F.~N., and Barton, D.~E.}
\newblock {\em Combinatorial Chance}.
\newblock Hafner Publishing Company, New York, 1962.

\bibitem{siva-dey-gamma_positive_excedances_alt_group}
{\sc Dey, H.~K., and Sivasubramanian, S.}
\newblock Gamma positivity of the {E}xcedance based {E}ulerian polynomial in
  positive elements of {C}lassical {W}eyl {G}roups.
\newblock {\em available at https://arxiv.org/abs/1812.02742\/} (2018), 20
  pages.

\bibitem{fulman-kim-lee-petersen_joint-distrib-descents-sign}
{\sc Fulman, J., Kim, G.~B., Lee, S., and Petersen, T.~K.}
\newblock On the joint distribution of descents and signs of permutations.
\newblock {\em https://arxiv.org/abs/1910.04258 arxiv preprint\/} (2019).

\bibitem{kahle-stump}
{\sc Kahle, T., and Stump, C.}
\newblock Counting inversions and descents of random elements in finite
  {C}oxeter groups.
\newblock {\em Mathematics of Computation 89}, 321 (2020), 437--464.

\bibitem{macmahon-book}
{\sc MacMahon, P.~A.}
\newblock {\em Combinatory {A}nalysis}.
\newblock Cambridge University Press, 1915-1916 (Reprinted by AMS Chelsea,
  2000).

\bibitem{mantaci-thesis}
{\sc Mantaci, R.}
\newblock {\em Statistiques {E}ul\'{e}riennes sur les {G}roupes de
  {P}ermutation}.
\newblock PhD thesis, Universit\'{e} Paris, 1991.

\bibitem{mantaci-jcta-93}
{\sc Mantaci, R.}
\newblock Binomial {C}oefficients and {A}nti-excedances of {E}ven
  {P}ermutations: {A} {C}ombinatorial {P}roof.
\newblock {\em Journal of Combinatorial Theory, Ser A 63\/} (1993), 330--337.

\bibitem{petersen-eulerian-nos-book}
{\sc Petersen, T.~K.}
\newblock {\em Eulerian Numbers}, 1st~ed.
\newblock Birkh{\" a}user, 2015.

\bibitem{reiner-descents-weyl}
{\sc Reiner, V.}
\newblock Descents and one-dimensional characters for classical {W}eyl groups.
\newblock {\em Discrete Mathematics 140\/} (1995), 129--140.

\bibitem{siva-exc-det}
{\sc Sivasubramanian, S.}
\newblock Signed excedance enumeration via determinants.
\newblock {\em Advances in Applied Math 47\/} (2011), 783--794.

\bibitem{siva-sgn_exc_hyp}
{\sc Sivasubramanian, S.}
\newblock Signed {E}xcedance {E}numeration in the {H}yperoctahedral group.
\newblock {\em Electronic Journal of Combinatorics 21(2)\/} (2014), P2.10.

\bibitem{siva-weyl_linear}
{\sc Sivasubramanian, S.}
\newblock Enumerating {E}xcedances with {L}inear {C}haracters in {C}lassical
  {W}eyl {G}roups.
\newblock {\em S{\'e}minaire Lotharingien de Combinatoire B74c\/} (2016), 15
  pp.

\bibitem{stembridge-perm-rep-weyl-grp-cohomology-toric-variety}
{\sc Stembridge, J.~R.}
\newblock Some permutation representations of {W}eyl groups associated with the
  cohomology of toric varieties.
\newblock {\em Advances in Mathematics 106\/} (1994), 244--301.

\end{thebibliography}
\end{document}